\documentclass[11pt]{article}
\usepackage{amsmath}
\usepackage{dsfont}
\usepackage{mathrsfs}
\usepackage{amsmath,amssymb}
\usepackage{amsfonts}
\usepackage{hyperref}
\usepackage{amsthm}
\usepackage{graphicx}
\usepackage{subfigure}
\usepackage{xcolor}
\renewcommand{\qed}{\hfill\small{$\square$}\normalsize}

\hfuzz=\maxdimen
\theoremstyle{definition}
\newtheorem{lemma}{Lemma}[section]

\newtheorem{proposition}[lemma]{Proposition}
\newtheorem{theorem}[lemma]{Theorem}
\newtheorem{corollary}[lemma]{Corollary}

\numberwithin{equation}{section}

\renewcommand{\qed}{\hfill\small{$\square$}\normalsize}

\DeclareFixedFont{\Acknowledgment}{OT1}{cmr}{bx}{n}{14pt}
\begin{document}

\title{\bf Kazdan-Warner equation on graph in the negative case}
\author{Huabin Ge}
\maketitle

\begin{abstract}
Let $G=(V,E)$ be a connected finite graph. In this short paper, we reinvestigate the Kazdan-Warner equation
$$\Delta u=c-he^u$$
with $c<0$ on $G$, where $h$ defined on $V$ is a known function. Grigor'yan, Lin and Yang \cite{GLY} showed that if the Kazdan-Warner equation has a solution, then $\overline{h}$, the average value of $h$, is negative. Conversely, if $\overline{h}<0$, then there exists a number $c_-(h)<0$, such that the Kazdan-Warner equation is solvable for every $0>c>c_-(h)$ and it is not solvable for $c<c_-(h)$. Moreover, if $h\leq0$ and $h\not\equiv0$, then $c_-(h)=-\infty$.

Inspired by Chen and Li's work \cite{CL}, we ask naturally:
\begin{center}
Is the Kazdan-Warner equation solvable for $c=c_-(h)$?
\end{center}
In this paper, we answer the question affirmatively. We show that if $c_-(h)=-\infty$, then $h\leq0$ and $h\not\equiv0$. Moreover, if $c_-(h)>-\infty$, then there exists at least one solution to the Kazdan-Warner equation with $c=c_-(h)$.
\end{abstract}


\section{Introduction}\label{Introduction}
It is well known that the following equation
\begin{equation*}\label{equ-smooth-simplify}
\Delta_gu=c-he^{u}
\end{equation*}
gives a description of the conformal deformation of the smooth metric $g$ on a $2$-dimensional closed Riemannian manifold $(M,g)$. Kazdan and Warner had given satisfying characterizations to the solvability of the above equation.

Grigor'yan, Lin and Yang \cite{GLY} first studied the corresponding equation on a connected finite graph $G=(V,E)$, where $V$ is the vertex set and $E$ is the edge set.
Denote $C(V)$ as the set of real functions on $V$. The discrete graph Laplacian $\Delta:C(V)\rightarrow C(V)$ is
\begin{equation*}
\Delta f_i=\frac{1}{\mu_i}\sum\limits_{j\thicksim i}\omega_{ij}(f_j-f_i)
\end{equation*}
for $f\in C(V)$ and $i\in V$, where $\mu:V\rightarrow(0,+\infty)$ is a fixed vertex measure, and $\omega:E\rightarrow(0,+\infty)$ is a fixed symmetric edge measure on $G$. Grigor'yan, Lin and Yang \cite{GLY} considered
\begin{equation}\label{def-KW-eq}
\Delta u=c-he^{u},
\end{equation}
where $c\in\mathds{R}$, and $h\in C(V)$. For any $f\in C(V)$, denote $\overline{f}$ as the average value of $f$ with respect to the measure $\mu$.
Let us summarize the results of Grigor'yan, Lin and Yang \cite{GLY}.
\begin{itemize}
  \item $c=0$ case. Assume $h\not\equiv0$, then the equation (\ref{def-KW-eq}) has a solution if and only if $h$ changes sign and $\overline{h}<0$.
  \item $c>0$ case. The equation (\ref{def-KW-eq}) has a solution if and only if $h$ is positive somewhere.
  \item $c<0$ case. If (\ref{def-KW-eq}) has a solution, then $\overline{h}<0$. On the contrary, if $\overline{h}<0$, then there exists a constant $-\infty\leq c_-(h)<0$ depending only on $h$ such that (\ref{def-KW-eq}) has a solution for any $c\in(c_-(h),0)$, but has no solution for any $c<c_-(h)$. Moreover, if $h\leq0$ and $h\not\equiv0$, then $c_-(h)=-\infty$ and hence (\ref{def-KW-eq}) always has a solution.
\end{itemize}
In the following we shall call the equation (\ref{def-KW-eq}) ``Kazdan-Warner equation" (on graph). One can see from Grigor'yan, Lin and Yang's results, when $c$ is nonnegative, the solvability of (\ref{def-KW-eq}) has been understood completely. However, in the case when $c$ is negative, one still needs to know:
\begin{center}
Can one solve the Kazdan-Warner's equation (\ref{def-KW-eq}) when $c=c_-(h)$?
\end{center}
The main purpose of this paper is to answer the above question. We prove
\begin{theorem}\label{thm-main}
Consider the Kazdan-Warner equation (\ref{def-KW-eq}) with $c<0$ and $\overline{h}<0$. Suppose that $c_-(h)<0$ is given as above. If $c_-(h)=-\infty$, then $h\leq0$ and $h\not\equiv0$. If $c_-(h)>-\infty$, then there exists at least one solution to (\ref{def-KW-eq}) with $c=c_-(h)$.
\end{theorem}

In the following of this paper, we prove Theorem \ref{thm-main} by using variational principles and the method of upper and lower solutions. We follow the approach pioneered by Chen and Li \cite{CL}, and Kazdan and Warner \cite{KW}.

\section{The proof of Theorem \ref{thm-main}}

\begin{lemma}\label{lem-1}
Consider the Kazdan-Warner equation (\ref{def-KW-eq}) with $c<0$. If it has a solution $u$, then the unique solution $\varphi$ to
\begin{equation}
(\Delta+c)\varphi=h
\end{equation}
satisfies $\varphi\geq e^{-u}$.
\end{lemma}
\begin{proof}
Using $e^x-1\geq x$ for any $x\in \mathds{R}$, we have $e^{u_i-u_j}-1\geq u_i-u_j$ and further
$$\omega_{ij}(e^{-u_j}-e^{-u_i})\geq e^{-u_i}\omega_{ij}(u_i-u_j).$$
Then for each $i\in V$, it follows
$$(\Delta e^{-u})_i\geq-e^{-u_i}(\Delta u)_i=-ce^{-u_i}+\Delta\varphi_i+c\varphi_i.$$
Hence
$$(\Delta+c)e^{-u}\geq(\Delta+c)\varphi.$$
Theorem 2.1 in \cite{Ge} tells us that the operator $-(\Delta+c)^{-1}$ preserves order. Hence we obtain
$\varphi\geq e^{-u}.$\qed
\end{proof}

\begin{lemma}\label{lem-2}
Given $f\in C(V)$, let $\varphi$ be the unique solution to the equation $(\Delta+c)\varphi=f$. Then
$$\lim\limits_{c\rightarrow -\infty}c\varphi=f.$$
\end{lemma}
\begin{proof}
Note the unique solution $\varphi$ depends on $c$. One can get the conclusion easily by orthogonally diagonalization method. \qed
\end{proof}

\begin{theorem}\label{thm-first-part}
If the Kazdan-Warner equation (\ref{def-KW-eq}) has at least a solution for any $c<0$, then $h\leq0$ and $h\not\equiv0$.
\end{theorem}
\begin{proof}
It's easy to see $h\not\equiv0$. We prove $h\leq0$ by contradiction. Assume there is a $i\in V$, such that $h_i>0$. Let $\varphi$ (depending on $c$) be the unique
solution to the equation $(\Delta+c)\varphi=h$. By Lemma \ref{lem-1}$, c\varphi_i$ tends to $h_i$ as $c$ tends to $-\infty$. This implies $\varphi_i<0$. However, by Lemma \ref{lem-2}, $\varphi_i\geq e^{-u_i}>0$, which contradicts $\varphi_i<0$. Hence $h\leq0$.\qed
\end{proof}

By Theorem \ref{thm-first-part} and Grigor'yan, Lin and Yang's results, we have
\begin{corollary}\label{cor-h<0}
Consider the Kazdan-Warner equation (\ref{def-KW-eq}) with $c<0$ and $\overline{h}<0$. Then $c_-(h)=-\infty$ if and only if $h\leq0$ and $h\not\equiv0$.
\end{corollary}

In the following of this paper, we assume $-\infty<c_-(h)<0$. This is equivalent to say $\overline{h}<0$ and there is a $i\in V$, such that $h_i>0$. We shall prove the Kazdan-Warner equation (\ref{def-KW-eq}) has a solution for $c=c_-(h)$.

Recall that $u_+$ is called a super solution to the Kazdan-Warner equation (\ref{def-KW-eq}) if
$$\Delta u_+-c+he^{u_+}\leq0,$$
while $u_-$ is called a sub solution to (\ref{def-KW-eq}) if
$$\Delta u_--c+he^{u_-}\geq0.$$
For any $f\in C(V)$, we define an integral of $f$ over $V$ with respect to the vertex weight $\mu$ by
$$\int_Vfd\mu=\sum\limits_{i\in V}\mu_if_i.$$
Thus $\overline{f}=\int_Vfd\mu\big/\int_Vd\mu$. Similarly, for any function $g$ defined on the edge set $E$, we define an integral of $g$ over $E$ with respect to the edge weight $\omega$ by
$$\int_Egd\omega=\sum\limits_{i\thicksim j}\omega_{ij}g_{ij}.$$

Now we choose a sequence of numbers $c_k$ such that
\begin{center}
$0>c_k>c_-(h)$, and $c_k\rightarrow c_-(h)$, as $k\rightarrow \infty$.
\end{center}
The main idea for the final proof is to find a solution $u^k$ to the following equation
\begin{equation}\label{equ-ck}
\Delta u=c_k-he^u,
\end{equation}
such that all $u^k$ are uniformly bounded. Therefore we can choose a subsequence of $u^k$ that tends to some $\hat{u}\in C(V)$. Taking limit in the above equation (\ref{equ-ck}), we show that $\hat{u}$ is a solution to the Kazdan-Warner equation (\ref{def-KW-eq}). The main difficulty is to choose such $u^k$. We shall prove the existence of such $u^k$ by the following four steps.

\textbf{Step 1.} Find a real number $A$, which is a sub solution to (\ref{equ-ck}) for any $k$.

In fact, for any $A$ satisfying $A<\ln\big(h_m^{-1}\sup\limits_{k}c_k\big)$, where $h_m<0$ is the minimum value of $h$, there holds
$$\Delta A-c_k+he^A>0.$$

\textbf{Step 2.} Find a $\psi^k$, which is a supper solution to (\ref{equ-ck}) for any $k$.

In fact, for each $k$, fix a $\widetilde{c_k}$ with $0>c_k>\widetilde{c_k}>c_-(h)$. By Grigor'yan, Lin and Yang's results, the equation $\Delta u=\widetilde{c_k}-he^u$ has at least a solution $\psi^k$, that is, $\Delta \psi^k=\widetilde{c_k}-he^{\psi^k}$. Then it follows
$$\Delta \psi^k-c_k+he^{\psi^k}<0,$$
that is, $\psi^k$ is a super solution to the equation (\ref{equ-ck}).

\begin{proposition}\label{claim-A<psi^k}
For each $k$, we have $A<\psi^k$.
\end{proposition}
\begin{proof}
Fix $k$. Since the graph $G$ is finite, we may choose a vertex $i\in V$, such that $\psi^k_i=\min\limits_{j\in V}\psi^k_j$. We just need to prove $A<\psi^k_i$. It's easy to see $\Delta\psi^k_i\geq0$. By
$$ h_i e^{\psi^k_i}=\widetilde{c_k}-\Delta\psi^k_i\leq\widetilde{c_k}<0$$
we see $h_i<0$. Further note
$$ h_i e^A> c_k>\Delta\psi^k_i+h_ie^{\psi^k_i},$$
we obtain
$$h_i(e^A-e^{\psi^k_i})>\Delta\psi^k_i\geq0,$$
hence $e^A<e^{\psi^k_i}$ and then $A<\psi^k_i$.\qed
\end{proof}

\textbf{Step 3.} Find a $u^k$ for each $k$, such that $A<u^k<\psi^k$, $u^k$ is a solution to the equation (\ref{equ-ck}), and $u^k$ minimize the following functional
$$I_k(f)=\frac{1}{2}\int_E|\nabla f|^2d\omega+c_k\int_Vfd\mu-\int_Vhe^fd\mu,\;\;A\leq f\leq\psi^k.$$
Since $\{f:A\leq f\leq \psi^k\}$ is a compact set of a finite dimensional vector space, $I_k(f)$ attains its minimum at some $u^k\in \{f:A\leq f\leq \psi^k\}$. We first show that
$$A<u^k<\psi^k.$$
In fact, if $u^k_i=A$ at some vertex $i\in V$, then
$$0\leq \nabla_iI_k(f)\big|_{f=u^k}=-\mu_i\left(\Delta u^k_i-c_k+h_ie^{u^k_i}\right).$$
Hence
$$\Delta u^k_i-c_k+h_ie^{u^k_i}\leq0.$$
Note $u^k\geq A$, hence
$$\Delta u_i^k=\frac{1}{\mu_i}\sum\limits_{i\thicksim j}\omega_{ij}(u_j^k-u_i^k)\geq0.$$
However,
$$\Delta u_i^k\leq c_k-h_ie^{u_i^k}=c_k-h_ie^A<\Delta A=0,$$
which is a contradiction. Therefore we obtain $A<u^k$. Similarly, if $u^k_i=\psi_i^k$ at some vertex $i\in V$,
then $0\geq \nabla_iI_k(f)\big|_{f=u^k}$, and it follows $\Delta u^k_i-c_k+h_ie^{u^k_i}\geq0$. Note $u^k\leq \psi^k$, hence
$$\Delta (\psi^k-u^k)_i=\frac{1}{\mu_i}\sum\limits_{i\thicksim j}\omega_{ij}\left(\big(\psi^k_j-u^k_j\big)-
\big(\psi^k_i-u^k_i\big)\right)\geq0.$$
However,
$$\Delta(\psi^k-u^k)_i=(\Delta \psi^k_i-c_k+h_ie^{\psi^k_i})-(\Delta u^k_i-c_k+h_ie^{u^k_i})<0,$$
which is a contradiction. Hence we obtain $u^k<\psi^k$. Then we show $u^k$ is a solution to the equation (\ref{equ-ck}). In fact, calculate the Euler-Lagrange equation of $I_k(f)$ at $u^k$, we obtain
$$0=\frac{d}{dt}\Big|_{t=0}I_k(u^k+t\phi)=\int_V(-\Delta u^k+c_k-he^{u^k})\phi d\mu$$
for any function $\phi$. Hence
$$-\Delta u^k+c_k-he^{u^k}=0,$$
which implies that $u^k$ is a solution to the equation (\ref{equ-ck}). Thus we finish this step.

\textbf{Step 4.} We show all $u^k$ are uniformly bounded. We just need to show all $u^k$ are uniformly bounded above. By the assumption of $c_-(h)>-\infty$, there exists at least one vertex $i\in V$ such that $h_i>0$. Since $u^k$ is the minimum point of $I_k$, then
$$0\leq\nabla^2_{ii}I_k(f)\big|_{f=u^k}=\sum\limits_{i\thicksim j}\omega_{ij}-h_i\mu_ie^{u^k_i}.$$
This leads to
$$e^{u^k_i}\leq\frac{1}{h_i\mu_i}\sum\limits_{i\thicksim j}\omega_{ij},$$
which implies that all $A\leq u^k_i\leq C_{h,G}$ are bounded above by some universal constant $C_{h,G}$ depending only on the information of $G$ and $h$. Now by
$$\frac{1}{\mu_i}\sum\limits_{i\thicksim j}\omega_{ij}(u_j^k-u_i^k)=\Delta u_i^k=c_k-h_ie^{u^k_i},$$
it's easy to see that for every vertex $j$ with $j\thicksim i$, $u^k_j$ is bounded above by some universal constant $C_{h,A,G}$. Since the graph $G$ is finite and connected, we derive $A\leq u_j^k\leq C_{h,A,G}$ for all $j\in V$ and all $k\in\mathds{N}$ by induction.

Since all $u^k$ are uniformly bounded, we may choose a subsequence $u^{k_n}$, such that
$$u^{k_n}\rightarrow \hat{u}$$
as $n\rightarrow\infty$. Let $n\rightarrow \infty$ in the following equation
$$\Delta u^{k_n}=c_k-he^{u^{k_n}},$$
we obtain
$$\Delta \hat{u}=c_-(h)-he^{\hat{u}}.$$
Thus the Kazdan-Warner equation (\ref{def-KW-eq}) has at least one solution $\hat{u}$ for $c=c_-(h)$.\\

\noindent \textbf{Acknowledgements:} The author would like to thank Professor Gang Tian and Yanxun Chang for constant encouragement. The author would also like to thank Professor Congming Li and Dr. Wenshuai Jiang for many helpful conversations. The research is supported by National Natural Science Foundation of China under Grant No.11501027.

Huabin Ge: hbge@bjtu.edu.cn

Department of Mathematics, Beijing Jiaotong University, Beijing 100044, P.R. China

\begin{thebibliography}{50}
\setlength{\itemsep}{-2pt} \small

\bibitem{CL} Chen, Wenxiong; Li, Congming. \emph{Gaussian curvature in the negative case}. Proc. Amer. Math. Soc. 131 (2003), no. 3, 741-744.
\bibitem{DL} Ding, Wei Yue; Liu, Jia Quan. \emph{A note on the problem of prescribing Gaussian curvature on surfaces}. Trans. Amer. Math. Soc. 347 (1995), no. 3, 1059-1066.
\bibitem{Ge} Ge, Huabin. \emph{p-th Kazdan-Warner equation on graph}. arXiv:1611.04902 [math.DG].
\bibitem{KW} Kazdan, Jerry L.; Warner, F. W. \emph{Curvature functions for compact 2-manifolds}. Ann. of Math. (2) 99 (1974), 14-47.
\bibitem{GLY} Grigor'yan, A.; Lin, Yong; Yang, Yunyan. \emph{Kazdan-Warner equation on graph}. Calc. Var. Partial Differential Equations 55 (2016), no. 4, Paper No. 92, 13 pp.
\end{thebibliography}
\end{document}